\documentclass{birkjour}
\usepackage[noadjust]{cite}
\usepackage{xcolor}
\RequirePackage[all]{xy}
\usepackage{amsmath,amsthm,amssymb,dsfont}
\usepackage{enumitem}
\usepackage[T1]{fontenc}
\usepackage[utf8]{inputenc}
\usepackage{enumitem}
\newtheorem{thm}{Theorem}

\newtheorem{lem}[thm]{Lemma}
\newtheorem{prop}[thm]{Proposition}
\theoremstyle{definition}
\newtheorem{defn}[thm]{Definition}
\theoremstyle{remark}
\newtheorem{rem}[thm]{Remark}

\newcommand{\BibTeX}{B\kern-0.1emi\kern-0.017emb\kern-0.15em\TeX}
\newcommand{\XYpic}{$\mathrm{X\kern-0.3em\raisebox{-0.18em}{Y}}$-$\mathrm{pic}\,$}

\newcommand{\cl}{C \kern -0.1em \ell}  

\DeclareMathOperator{\Erg}{Erg}

\newcommand{\OO}{\mathcal{O}}

\newcommand{\B}{\mathcal B}

\newcommand{\C}{\mathbb C}

\newcommand{\N}{\mathbb N}

\newcommand{\supp}{\text{supp}}

\newcommand{\ed}{\end{document}}
\begin{document}

%
%
%
%
%
%
%
%
%

\title
 {$L^p$-Cuntz algebras and spectrum of weighted composition operators}
 \author{Krzysztof Bardadyn}
 \address{Faculty of Mathematics, University of Bia\l ystok\\
   ul.~K.~Cio{\l}kowskiego 1M, 15-245\\
   Bia{\l}ystok, Poland} \email{kbardadyn@math.uwb.edu.pl}

\thanks{This work was supported by the National Science Center (NCN) grant no.~2019/35/B/ST1/02684. I thank Bartosz Kwa\'sniewski  
for his constant support and motivation, as well as careful reading of this manuscript and suggested improvements.}
%
%


\begin{abstract}
  Let $p\in [1,\infty)$.  We define an $L^p$-operator algebra crossed
  product by a transfer operator for the topological Bernoulli shift
  $\varphi$ on $X=\{1,...,n\}^{\N}$, and we prove it is isometrically
  isomorphic to the $L^p$-analog $\OO_n^p$ of the Cuntz algebra
  introduced by Phillips.  As an application we prove that the
  spectrum of the associated `abstract weighted shift operators' $aT$,
  $a\in C(X)$, is a disk with radius given by the formula
  \begin{equation*}
    r(aT)=\max_{\mu\in\Erg(X,\varphi)} \exp\left( \int_X \ln
      (|a|\varrho^{\frac{1}{p}})d\mu + \frac{h_\varphi(\mu)}{p} \right)
  \end{equation*}
  where $\varrho$ is a potential associated to the transfer operator,
  and $h_\varphi(\mu)$ is Kolmogorov-Sinai entropy.  This generalizes
  classical results for $p=2$.
\end{abstract}
\label{page:firstblob}
\maketitle

\section{Introduction}
Cuntz algebras $\OO_n$, $n \geq 2$, play a predominant role in the
theory of $C^*$-algebras and provide important links to other branches
of mathematics such as dynamical systems, representation theory,
wavelet theory, and others, see e.g. \cite{BJ}. In \cite{PhLp1}
Phillips introduced $L^p$-analogues of Cuntz algebras $\OO_n^p$ that
are complex Banach algebras represents on $L^p$-spaces for
$p\in [1,\infty)$.  By definition $\OO_n^p$ is a Banach algebra which
is universal for \emph{$L^p$-Cuntz families}, i.e. families of
contractive operators $T_{i}, S_{i}$, $i=1,..., n$, acting on an
$L^p$-space $L^p_{\mu}(\Omega)$, for some measure space, satisfying
\begin{equation*}
\sum_{i=1}^{n} T_{i}S_{i}=1, \qquad S_iT_i=1\quad  \text{and} \quad  T_iS_i \text{ is hermitian},  \quad \text{for  $i=1,...,n$}.
\end{equation*}
In fact, for any such family, the Banach subalgebra $B(T_i, S_i:i=1,...,n)$ of $B(L^p(\mu))$ generated by $T_{i}, S_{i}$, $i=1,... n$ is naturally isometrically isomorphic to $\OO_n^p$,
and also $\OO_n^p$ is  simple and purely infinite, see~\cite{PhLp1}, \cite{PhLp2a}.  For $p=2$ we necessarily have $S_i=T_i^*$ and $\OO_n^2=\OO_n$
 is the standard Cuntz $C^*$-algebra~\cite{Cuntz}. For $p\neq 2$, operators  $T_{i}, S_{i}$ are necessarily \emph{spatial partial isometries} by generalized Lamperti's Theorem, see~\cite[Theorem 2.26]{BKM}. 

Cuntz $C^*$-algebras, and their generalizations, have natural models as crossed products by transfer operators (often called Exel's crossed product~\cite{Exel_crossed}), which give link Cuntz algebras with thermodynamical formalism. In this  note we define an $L^p$-operator algebra crossed product by a transfer operator, and we  apply this construction to spectral analysis of weighted composition operators in the context of  Cuntz-algebras. This generalizes the motivating example from \cite{bar-kwa}, to the case of Banach algebras and operators acting on $L^p$-spaces (rather than Hilbert spaces).


More specifically, let $(X,\varphi)$ be a \emph{topological Bernoulli shift} on the alphabet with $n$ letters, i.e. $X:=\{1,...,n\}^{\N}$ is a  Cantor space with product topology, 
and $\varphi(x_1,x_2,...):=  (x_2,x_{3}...)$ for $x=(x_1,x_2,...)\in X$. We fix a \emph{transfer operator} for $\varphi$, i.e. a positive operator $L:C(X)\to C(X)$ acting on the space of continuous functions on $X$  satisfying 
\begin{equation*}
L(b \cdot (a \circ \varphi))= L(b) a, \qquad  \text{ for  }a,b\in C(X).
\end{equation*} Then  it is necessarily  a \emph{Ruelle-Perron-Frobenius operator} of the form
\begin{equation*}
L(a)(y)=\sum_{x\in\varphi^{-1}(y)}\varrho(x)a(x), \qquad a\in C(X),
\end{equation*}
where $\varrho:X\to [0,+\infty)$ is a continuous function. One refers to $\varrho$ as to the \emph{potential} of the system. We will assume that $\varrho >0$ and $L$ is unital and so
$\sum_{x\in\varphi^{-1}(y)}\varrho(x)=1$ for all $x\in X$ (such operators are algebraically characterized in \cite[Proposition 2.13]{bar-kwa}). 
By
the Schauder-Tychonoff fixed point theorem, there is always  a probability Borel measure $\mu$ on $X$ such that $L^*(\mu)=\mu$ (when $\ln \varrho$ is H\"older continuous, such $\mu$ is unique and it is a \emph{Gibbs measure} for $\varphi$ and $\ln \varrho$). Then  $\mu$ is $\varphi$-invariant and its support is the whole $X$, see \cite[Proposition 3.8]{bar-kwa}. 
Therefore, for any $p\in [1,\infty)$  the composition operator $T_{\varphi} f= f\circ \varphi$ is an isometry on $L^p_{\mu}(X)$ 
and  we may identify elements of $C(X)$ with operators of multiplication acting on $L^p_{\mu}(X)$.
We write $\pi:C(X)\to  B(L^p_{\mu}(X))$ for the corresponding isometric representation.
We fix  $q\in (1,\infty]$ with $1/p+1/q=1$ (where $1/\infty=0$ by convention) 
and we put $X_{i}:=\{(i,x_2,x_3,...) \in X\}$ for $i=1,..., n$.
\begin{prop}\label{prop:motivation}
With the above notation,  $S_{\varphi}f(y):=\sum_{x\in\varphi^{-1}(y)}\varrho(x)f(x)$ defines a contractive operator on  $L^p_{\mu}(X)$ such that
\begin{equation}\label{eq:relations_concrete}
S_{\varphi}\pi(a)T_{\varphi} =\pi(L(a)), \ \ T_{\varphi}\pi(a)=\pi(a\circ\varphi)T_{\varphi},\qquad a\in C(X),
\end{equation}
and operators $T_{i}= \pi(\varrho^{-\frac{1}{p}}\mathds{1}_{X_i})T_\varphi$ and $S_i= S_{\varphi}\pi(\varrho^{-\frac{1}{q}}\mathds{1}_{X_i})$ form an $L^p$-Cuntz family.
The Banach algebra   $B(\pi(C(X)),T_\varphi,S_\varphi)$ generated by $\pi(C(X))$, $T_\varphi$, $S_\varphi$ is the Cuntz algebra $\OO_n^p$, i.e. $B(\pi(C(X)),T_\varphi,S_\varphi)=B(T_i, S_i:i=1,...,n)$.
\end{prop} 
\begin{proof}
Using  Jensen's inequality and  $L^*(\mu)=\mu$, for any $f\in L^p_{\mu}(X)$ we get
\begin{equation*}
\|S_{\varphi}f\|^p= \int_{X}  \Big|\sum_{x\in\varphi^{-1}(y)}\varrho(x)f(x)\Big|^p\, d\mu\leq \int_{X} \sum_{x\in\varphi^{-1}(y)}\varrho(x)|f(x)|^p\, d\mu= \|f\|_{p}^{p}.
\end{equation*}
Hence $S_{\varphi}$ is well defined and contractive. Now the displayed relations \eqref{eq:relations_concrete} are easily verified. 
The restrictions $\varphi_i:=\varphi|_{X_i}:X_i\to X$, $i=1,...n$, are homeomorphisms and by \cite[Proposition 3.5]{bar-kwa} we have
 \begin{equation*}
\varrho(\varphi_i^{-1}(x))=\frac{d\mu\circ\varphi^{-1}_i}{d\mu}(x), \qquad x\in X.
\end{equation*}
We use it to show that $T_i= \pi(\varrho^{-\frac{1}{p}}\mathds{1}_{X_i})T_\varphi$ are isometries. 
 For any $f\in L^p(\mu)$,  using also abstract change of variables, we get
\begin{align*}
\|T_if\|_p^p
&=\int_{X_i}\varrho(x)^{-1}|f(\varphi_i(x))|^p d\mu=\int_{X_i}\varrho(\varphi_i^{-1}(x))^{-1}|f(x)|^p d\mu\circ\varphi_i^{-1} \\
&=\int_{X}|f(x)|^p  \frac{d\mu}{d\mu\circ\varphi^{-1}_i}(x) d\mu\circ\varphi_i^{-1}
=\|f\|^p.
\end{align*}
Let us note that $S_i f(x)=\varrho(\varphi^{-1}_i(x))^{\frac{1}{p}}f(\varphi^{-1}_i(x))$. Therefore we readily get that 
$S_iT_i=1$ and  $
T_iS_i= \pi(\mathds{1}_{X_i})
$. These are hermitian projections as mutliplication operators by real functions.  
In particular, $\sum_{i=1}^nT_iS_i=\sum_{i=1}^n  \pi(\mathds{1}_{X_i})=\pi(\mathds{1}_{X})=1$, 
and $S_i$ are contractions as  for any $f\in L^p_{\mu}(X)$, using that $T_i$ is an isometry, we get
\begin{equation*}
\|S_if\|_p=\|T_iS_if\|_p=\|\mathds{1}_{X_i}f\|_p\leq\|f\|_p.
\end{equation*}
Hence $T_{i}, S_{i}$, $i=1,..., n$, form an $L^p$-Cuntz family and so $B(T_i, S_i:i=1,...,n)\cong \OO_n^p$.
Clearly, we have $B(T_i, S_i:i=1,...,n)\subseteq B(\pi(C(X)),T_\varphi,S_\varphi)$, and thus we only need 
to prove the reverse inclusion. We first  show that $\pi(C(X))\subseteq B(T_i, S_i:i=1,...,n)$. 
For any $(i_1,...,i_N)\in\{1,...,n\}^N$ we have 
\begin{equation*}
T_{i_1}...T_{i_N}S_{i_N}...S_{i_1}=\pi(\mathds{1}_{X_{i_1,...,i_N}})\in B(T_i, S_i:i=1,...,n),
\end{equation*} 
where 
\begin{equation}\label{eq:basic_sets}
X_{i_1,...,i_N}=\{(i_1,...,i_N,x_{N+1},...)\in X\}.
\end{equation}
Since sets $X_{i_1,...,i_N}$ form a clopen base of $X$, the Stone–Weierstrass  theorem implies that $\pi(C(X))\subseteq B(T_i, S_i:i=1,...,n)$. 
Now we see that operators
\begin{equation*}
T_{\varphi}=\sum_{i=1}^n\pi(\varrho^\frac{1}{p})T_i,\qquad  S_{\varphi}=\sum_{i=1}^nS_i\pi(\varrho^\frac{1}{p}) 
\end{equation*}
are in $B(T_i, S_i:i=1,...,n)$. 
\end{proof}

The above proposition motivates the following definition. 
\begin{defn}
A \emph{covariant representation} of the transfer operator $L$ on a space $L^p_{\mu}(\Omega)$ for some measure space $(\Omega,\Sigma,\mu)$ we mean a triple $(\pi, T, S)$ where 
$\pi:C(X)\to B(L^p_{\mu}(\Omega))$ is a unital contractive homomorphism and $T, S\in B(L^p_{\mu}(\Omega))$ are contractive operators satisfying
\begin{equation}\label{eq:covariant_relation}
S\pi(a)T =\pi(L(a)),\qquad  T\pi(a)=\pi(a\circ\varphi)T, \qquad a\in C(X)
\end{equation}
and such that  $\pi(\varrho^{-\frac{1}{p}}\mathds{1}_{X_{i}})T$, $S\pi(\varrho^{-\frac{1}{q}}\mathds{1}_{X_i})$, $i=1,...,n$, are contractive and
\begin{equation}\label{eq:covariant_relation2}
\sum_{i=1}^{n} \pi(\varrho^{-\frac{1}{p}}\mathds{1}_{X_i})TS\pi(\varrho^{-\frac{1}{q}}\mathds{1}_{X_i})=1.
\end{equation}
We denote by $B(\pi(C(X)), T, S)$ the Banach subalgebra of $B(L^p_{\mu}(\Omega))$ 
generated by $\pi(C(X))$, $T$ and  $S$.
\end{defn}
For $p=2$ the above definition agrees with the one in \cite{Exel_crossed},  see \cite{kwa1}, \cite{BKL} and Remark \ref{rem:p=2} below.
For $p\neq 2$, as we explain in Section \ref{sec:two},  covariant representations are given by multiplication operators and weighted composition operators. 
 In Section \ref{sec:two_and_half}   we  prove the following theorem.
\begin{thm}\label{thm:first_result} 
For any (non-zero) covariant representation $(\pi, T, S)$ of $L$ on an $L^p$-space $L^p_{\mu}(\Omega)$, $p\in [1,\infty)$, 
the operators $T_{i}= \pi(\varrho^{-\frac{1}{p}}\mathds{1}_{X_i})T$ and $S_i= S_{\varphi}\pi(\varrho^{-\frac{1}{q}}\mathds{1}_{X_i})$ form an $L^p$-Cuntz family, and 
\begin{equation*}
B(\pi(C(X)), T, S)=B(T_i, S_i:i=1,...,n)\cong\OO_n^p.
\end{equation*}

\end{thm}
As an application we get the following description of spectra of 
abstract weighted shift operators  on $L^p$-spaces associated to the transfer operator $L$. 
The proof is given in Section \ref{sec:three}.

\begin{thm}\label{thm:second_result}  For any (non-zero) covariant representation $(\pi, T, S)$ of $L$ on an $L^p$-space
and for any $a\in C(X)$ the spectrum of $\pi(a)T$ is a disk
\begin{equation*}
\sigma(\pi(a)T)=\{ z\in \C: |z|\leq r(\pi(a)T)\}
\end{equation*}
and the spectral radius is given by the variational principle
\begin{equation*}
r(\pi(a)T)=\max_{\mu\in\Erg(X,\varphi)} \exp\left( \int_X \ln (|a|\varrho^{\frac{1}{p}})d\mu + \frac{h_\varphi(\mu)}{p} \right),
\end{equation*}
where $\Erg(X,\varphi)$ is the set of all $\varphi$-ergodic measures on $X$, and $h_\varphi(\mu)$ is the Kolmogorov-Sinai entropy of $\varphi$ with respect to $\mu$.
\end{thm}
\begin{rem} For $p=1$ the expression for $r(\pi(a)T)$ is nothing but the  topological pressure of $\varphi$ with the potential $\ln (|a|\varrho)$
For more about the history of such variational formulas and related results, see \cite{ABL}, \cite{bar-kwa}.
\end{rem}
\section{Description of covariant represenations for $p\neq 2$}\label{sec:two}

Assume that $p\in [1,\infty)\setminus\{2\}$ and let $(\pi,T,S)$ is a covariant representation on $L^p_{\mu}(\Omega)$ for a measure space $(\Omega,\Sigma,\mu)$. 
Without changing the Banach space $L^p_{\mu}(\Omega)$  we may assume that the measure space is \emph{localizable}, see \cite{Fremlin}, \cite{BKM}, \cite{cgt}. 
Then representation $\pi:C(X)\to B(L^p_{\mu}(\Omega))$ is necessarily given by \emph{multiplication operators}, 
see \cite[Theorem 2.12]{BKM}. That is,  a unital $*$-homomorphism $\pi_0:C(X)\to  L^{\infty}_{\mu}(\Omega)$ such that 
\begin{equation*}
(\pi(a)f)(\xi)=\pi_0(a)(\xi)f(\xi),\qquad   \xi \in L^p_{\mu}(\Omega), \,\, \omega\in \Omega.
\end{equation*}
We now describe the operators  $T$ and $S$ as weighted composition operators.

By the first relation in \eqref{eq:covariant_relation} we get  $ST=S\pi(1)T=\pi(L(1))=\pi(1)=1$. Since $S$ and $T$ are contractions, it follows that $T$ is an isometry. (In particular, $S$ and $T$ are
partial isometries in the sense of Mbekhta \cite{Mbekhta}.) By the  \emph{Banach-Lamperti Theorem}, see \cite[Theorem 3.1]{Lamperti},  $T$ is necessarily a (generalized) weighted composition operator of the form $hT_{\Phi}$. 
Namely, the formula  $\Phi(A):=\supp(T \mathds{1}_{A})$ defines a set  monomorphism $\Phi: \Sigma\to \Sigma$, i.e. identifying sets that differ by $\mu$-null sets, we have

\begin{enumerate}
\item[1)]\label{enu:set_morphism_axiom_a1} $A\cap B=\emptyset\Rightarrow \Phi(A)\cap\Phi(B)=\emptyset$,
\item[2)]\label{enu:set_morphism_axiom_a2} $\Phi\big(\bigsqcup_{n=1}^\infty A_n\big)=\bigsqcup_{n=1}^\infty\Phi(A_n)$,
\item[3)]\label{enu:set_morphism_axiom_a3} $\mu(A)=0\iff\mu(\Phi(A))=0$,
\end{enumerate}
where $A,B, A_n \in \Sigma$. 
Then  $\Phi(\Sigma)$ is a $\sigma$-algebra and  $\mu\circ\Phi^{-1}(\Phi(A)):=\mu(A)$  defines a measure on $\Phi(\Sigma)$ equivalent to $\mu|_{\Phi(\Sigma)}$. 
Any  set monomorphism gives rise to 
a linear operator $T_{\Phi}$ on the space of measurable functions where
\begin{equation*}
T_{\Phi}\mathds{1}_{A}=\mathds{1}_{\Phi(A)}, \qquad A\in \Sigma,
\end{equation*}
see \cite[Proposition 5.6]{PhLp1}. Moreover, there is a  measurable function   $h:\Omega\to\C$,  with $\supp(h)=\Phi(\Omega)$, such that $T=hT_{\Phi}$, i.e.
\begin{equation*}
T\xi (\omega) = h(\omega) T_{\Phi}(\xi)(\omega). \qquad  \xi \in L^p_{\mu}(\Omega), \,\, \omega\in \Omega.
\end{equation*}
Relation \eqref{eq:covariant_relation2} implies that $\Phi(\Omega)=\Omega$ and $\supp(h)=\Omega$. 
 The function $h$ satisfies 
\begin{equation*}
E^{\Phi(\Sigma)}(|h|^p)=\frac{d\mu\circ\Phi^{-1}}{d\mu},
\end{equation*}
 where $E^{\Phi(\Sigma)}$ denotes the conditional expectation with respect to the $\sigma$-algebra $\Phi(\Sigma)\subseteq \Sigma$. 
This condition characterises isometries among weighted composition operators. The  Radon-Nikodym derivative exists because the measure space is localizable. 

Now we turn to  description of $S$. Note that $TS$ is a contractive projection onto the range  of $T$, that we denote by $M\subseteq L^p_{\mu}(\Omega)$. 
The \emph{Ando Theorem} that  describes contractive projections on $L^p$-spaces, see \cite[Theorem 2]{Ando} which generalizes to arbitrary measures by \cite{Bernau_Lacey_Elton}. It implies that  
\begin{equation*}
TS \xi=\frac{h}{E^{\Phi(\Sigma)}(|h|^p)}E^{\Phi(\Sigma)}\left(\frac{|h|^{p}}{h} \xi\right)
\end{equation*}
(for $p=1$ we use the fact that $\Phi(\Omega)=\Omega$). 
Denoting by $T^{-1}:M\to L^p_{\mu}(\Omega)$ the inverse of $T: L^p_{\mu}(\Omega)\to M$, we have $S=T^{-1} TS$. 
Therefore
\begin{equation*}
S\xi =\frac{d\mu\circ\Phi}{d\mu} T_{\Phi^{-1}} E^{\Phi(\Sigma)}\left(\frac{|h|^{p}}{h} \xi\right).
\end{equation*}
\section{Proof of Theorem \ref{thm:first_result} }\label{sec:two_and_half}

To show that operators $T_{i}= \pi(\varrho^{-\frac{1}{p}}\mathds{1}_{X_i})T$ and $S_i= S_{\varphi}\pi(\varrho^{-\frac{1}{q}}\mathds{1}_{X_i})$, $i=1,...,n$ form a Cuntz family note that the relation $\sum_{i=1}^{n} T_{i}S_{i}=1$ is equivalent to~\eqref{eq:covariant_relation}. Furthermore, since $L(\varrho^{-1}\mathds{1}_{X_i})=1$, we get
\begin{equation*}
S_i T_i=S\pi(\varrho^{-\frac{1}{q}}\mathds{1}_{X_i})\pi(\varrho^{-\frac{1}{p}}\mathds{1}_{X_i})T=S\pi(\varrho^{-1}\mathds{1}_{X_i})T=L(\varrho^{-1}\mathds{1}_{X_i})=1.
\end{equation*} 
Using that $\pi(\mathds{1}_{X_i})T_j=\delta_{i,j}T_j$ and $\sum_{i=1}^nT_iS_i=1$ we obtain that 
\begin{equation*}
T_i S_i=\pi(\mathds{1}_{X_{i}}).
\end{equation*} 
Therefore $T_iS_i$ are hermitian operators by   \cite[Proposition 2.11]{BKM}.
Hence  $T_i,S_i$, $i=1,...,n$ is an $L^p$-Cuntz family and therefore $B(T_i,S_i: i=1,...,n)\cong\OO^p_n$.
 
It remains to prove that $B(S_i, T_i: i=1,...,n)=B(\pi(C(X)),T,S)$. The inclusion $\subseteq $ is obvious.  
We prove the reverse inclusion. Recall the notation \eqref{eq:basic_sets}.
Using the relation $Ta=(a\circ\varphi)T$, $a\in C(X)$, we get
\begin{equation*}
T_{i_1}T_{i_2}S_{i_2}S_{i_1}=T_{i_1}\pi(\mathds{1}_{X_{i_2}})S_{i_1}=\pi(\mathds{1}_{\varphi^{-1}(X_{i_2})})\pi(\mathds{1}_{X_{i_1}})=\pi(\mathds{1}_{X_{i_1 i_2}}).
\end{equation*} 
Proceeding by induction one gets that $T_{i_1}...T_{i_N}S_{i_N}...S_{i_1}=\pi(\mathds{1}_{X_{i_1...i_N}})$ for any $N\in\N$, $i_1,...,i_N\in\{1,...,n\}$. 
Since functions of the form $\mathds{1}_{X_{i_1...i_N}}$ are linearly dense in $C(X)$, we obtain that $\pi(C(X))\subseteq B(T_i,S_i: i=1,...,n)$. 
Thus we also get
\begin{equation*}
T=\pi(\varrho^{\frac{1}{p}})\sum_{i=1}^{n}T_i,\,\,  S=\sum_{i=1}^{n}S_i\pi(\varrho^{\frac{1}{q}})\in B(T_i,S_i: i=1,...,n).
\end{equation*} 
Therefore $B(\pi(C(X)),T,S)=B(T_i,S_i: i=1,...,n)\cong\OO^p_n$.
\begin{rem}\label{rem:p=2}
If $p=2$, so that $L^2_\mu(\Omega)$ is a Hilbert space, then  $S=T^*$ and  $S_i=T_i^*$ for $i=1,...,n$, by \cite[Corollary 3.3]{Mbekhta}. 
Thus the relations we consider become the standard ones used in the context of $C^*$-algebras. 
\end{rem}
\begin{rem}
Assume $p\neq 2$. It follows from the considerations in this and previous section, it follows that putting $\Omega_i=\supp(\pi_0(\mathds{1}_{X_i}))$, $i=1,...,n$,
 we get a decomposition $\Omega:=\bigsqcup_{i=1}^{n} \Omega_i$ of $\Omega$  and 
each $T_{i}:L^p_\mu(\Omega)\to L^p_\mu(\Omega_i)$ is an invertible isometry whose inverse is given by
$S_i$. Moreover, putting
\begin{equation*}
\Phi_{i}(A):=\Phi(A) \cap \Omega_{i}, \qquad A\in \Sigma,
\end{equation*}
$\Phi_{i}:\Sigma \to \Sigma$ is a set monomorphism  with $\Phi_{i}(\Sigma):=\{A\in \Sigma: A\subseteq \Omega_{i}\}$,
and  
\begin{equation*}
T_{i}=\pi(\varrho^{-\frac{1}{p}}) h T_{\Phi_i}, \qquad  S_{i}= T_{\Phi_i^{-1}}\pi(\varrho^{\frac{1}{p}}) h^{-1} \mathds{1}_{\Omega_i} 
\end{equation*}
are spatial partial isometries as introduced in \cite{PhLp1}.


\end{rem}

\section{Proof of Theorem \ref{thm:second_result}}\label{sec:three}

Let $(\pi, T,S)$ be a covariant representation of $L$ on the space  $L^p_\mu(\Omega)$ for a measure space $(\Omega,\Sigma,\mu)$. Let $a\in C(X)$. We  show first that the spectrum of $\pi(a)T$ is invariant under rotation around zero. As we proved in Theorem \ref{thm:first_result}, the algebra $B(\pi(C(X)),S,T)$ is isomorphic with $\OO_n^p$. Phillips showed in \cite[Proposition 5.9]{PhLp2a} that $\OO_n^p$ has a gauge circle action $\tau: \mathbb{T} \to \text{Aut}(\OO_n^p)$, which in therms of $\pi(C(X)),S,T$ is determined by the following relations  
\begin{equation*}
\tau_z|_{\pi(C(X))}=\text{id}|_{\pi(C(X))}, \quad \tau_z (T)=z T, \quad \tau_z (S)=z^{-1} S, \qquad z\in \mathbb{T}.
\end{equation*}
Now suppose that $\lambda$ is in the resolvent set $\C\setminus\sigma(\pi(a)T)$. Then for any $z\in \mathbb{T}$ the operator $\tau_z (\pi(a)T-\lambda\mathds{1})$ is invertible and since
\begin{equation*}
\tau_z(\pi(a)T-\lambda\mathds{1})=\pi(a)zT-\lambda\mathds{1}=z(\pi(a)T-z^{-1}\lambda \mathds{1})
\end{equation*}
it follows that $\pi(a)T-z^{-1}\lambda \mathds{1}$ is invertible. Therefore $z^{-1}\lambda\in\C\setminus\sigma(\pi(a)T)$, which means that $\sigma(\pi(a)T)$ is invariant under rotation around zero. It implies $\sigma(\pi(a)T)=\bigsqcup_{i\in I}D(r_i, R_i)$ where $D(r_i, R_i)=\{ z\in\C: r_i\leq |z|\leq R_i\}$. 
By  \cite[Proposition 3.3]{bar-kwa} (the argument works in $L^p$-case) the operator $\pi(a)T$ is not invertible  and so $0\in\sigma(\pi(a)T)$. 
To prove that $\sigma(\pi(a)T)$ is a disc it is sufficient to show that $\sigma(\pi(a)T)$ is connected. 
To do this we use Riesz projections.  Assume that on the contrary that $\sigma(\pi(a)T)$ contains 
$D(r,R)$ where $0<r\leq R$. By scaling $\pi(a)T$ if necessary, we may assume that $1<r$ and  $\sigma(\pi(a)T)\cap \mathbb{T}=\emptyset$. The \emph{Riesz projection} of $\pi(a)T$ corresponding to the unit disk is  given by 
\begin{equation*}
P=\frac{1}{2\pi i}\int_{\mathbb{T}}(\pi(a)T-z\mathds{1})^{-1}dz.
\end{equation*}
By construction $P$ is a  bounded idempotent, i.e. $P^2=P$, such that:
\begin{enumerate}
\item[1)]\label{enu:riesz_projection_commutes} $P$ commutes with $\pi(a)T$;
\item[2)]\label{enu:riesz_projection_range}  $\sigma(\pi(a)T|_{PL^{p}_{\mu}(\Omega)})= \sigma(\pi(a)T)\cap D(0,1)$ - the spectrum of  $\pi(a)T$ restricted 
to the range of $P$ is  the part of $\sigma(\pi(a)T)$ which lies in $D(0,1)$;
\item[3)]\label{enu:riesz_projection_range} 
$\sigma(\pi(a)T|_{(1-P)L^{p}_{\mu}(\Omega)})= \sigma(\pi(a)T)\setminus D(0,1)$ - the spectrum of $\pi(a)T$ restricted 
to the kernel of $P$ is the part of $\sigma(\pi(a)T)$ outside $D(0,1)$.
\end{enumerate}
Using that   $\pi(a)T$ is a weighted isometry  we get: 
\begin{lem}\label{lem"'Riesz_projector}
The above projection $P$  commutes with elements of $\pi(C(X))$  and the restriction $\pi(a)|_{\ker(P)}:\ker(P)\to\ker(P)$ is invertible.
\end{lem}
\begin{proof}
The proof of \cite[Lemma 6.3]{bar-kwa} is valid in the $L^p$-case.
\end{proof}
Note that the Bernoulli shift $\varphi$ is \emph{topologically free}, i.e. the set $\{x\in X:\varphi^n(x)=x\text{ for some }n\in\N\}$ has an empty interior.
Hence using groupoid methods it can be shown that $\pi(C(X))$ is a maximal abelian subalgebra of $B(\pi(C(X)),T,S)$, see \cite{BKM2} or \cite{cgt}.
Combining this with Lemma \ref{lem"'Riesz_projector} we get that $P\in  \pi(C(X))$. 
Therefore $P= \pi(\mathds{1}_{Y_1})$ and $(1-P)=\pi(\mathds{1}_{Y_2})$ where $X=Y_1\sqcup Y_2$ is a decomposition into disjoint clopen sets. 
By our assumptions both $Y_1$ and $Y_2$ are non-empty ($P$ and $1-P$ are non-zero). 
By Lemma \ref{lem"'Riesz_projector}, $\pi(a\mathds{1}_{Y_2})$ is invertible as an element in $\B(\ker P)$.
This means that $a(y)\neq 0$ for $y\in Y_2$. Also since 
$(1-P)=\pi(\mathds{1}_{Y_2})$ commutes with $\pi(a)T$ and $\pi(\overline{a})$ it also commutes with $\pi(|a|^2)T$. Using this  we get
\begin{equation*}
\pi(L(|a|^2\mathds{1}_{Y_2}))=S \pi(|a|^2\mathds{1}_{Y_2}) T=S \pi(|a|^2) T \pi(\mathds{1}_{Y_2})=\pi(L(|a|^2)\mathds{1}_{Y_2})
\end{equation*}
Therefore $L(|a|^2\mathds{1}_{Y_2})(y)\neq 0$ only if $y \in Y_2$, which implies that $\varphi(Y_2)\subseteq Y_2$
But the only  non-empty open set forward invariant under the Bernoulli shift $\varphi$ is $X$. Hence $Y_2=X$ which leads to the contradiction with $Y_1\neq\emptyset$.

Thus we proved that the spectrum of $aT$ is a disc.  To complete the proof we compute the
spectral radius of $aT$. By Theorem \ref{thm:first_result}, $B(\pi(C(X)),T,S)$ does not depend on a choice of a covariant representation of $L$. In particular,
 we may assume that $T=T_\varphi$, $S=S_\varphi$ and $\pi$ are as in Proposition \ref{prop:motivation} for a probability Borel measure $\mu$ on $X$ with $L^*(\mu)=\mu$. 
Then for any $N\in\N$ and $f\in L^p_\mu(X)$ we have
\begin{equation*}
\int_X |(\pi(a)T)^Nf|^pd\mu=\int_X \prod_{i=0}^{N-1}|a\circ\varphi^i|^p|f\circ\varphi^N|^pd\mu=\int_X (\pi(|a|^p)T)^N|f|^pd\mu.
\end{equation*}
Hence $\|(\pi(a)T)^N\|^p_{L^p_\mu}=\|\pi(|a|^p)|T)^N\|_{L^1_\mu}$. Using the Gelfand formula for spectral radius we obtain 
\begin{equation*}
r(\pi(a)T)^p= r(\pi(|a|^p) T).
\end{equation*}
 Note that $S|a|^p:L^\infty_\mu(X)\to L^\infty_\mu(X)$ is adjoint to $\pi(|a|^p)T: L^1_\mu(X)\to L^1_\mu(X)$ and therefore their spectral radii coincide. 
$S|a|^p:L^\infty_\mu(X)\to L^\infty_\mu(X)$ is an extension of the weighted transfer operator $L|a|^p:C(X)\to C(X)$.
As these operators are positive their norms  (and the norms of their powers)  realize at the unit $\mathds{1}_{X}$. 
Thus their spectral radii also coincide. Concluding we get 
\begin{equation*}
r(\pi(a)T)^p=r(S|a|^p)=r(L|a|^p).
\end{equation*}
The following formula for the spectral radius of $L|a|^p:C(X)\to C(X)$
\begin{equation*}
r(L|a|^p)=\max_{\mu\in\Erg(X,\mu)}\exp\big(\int_{X}\ln(\varrho|a|^p) d\mu+h_{\varphi}(\mu).
\end{equation*}
goes back to the original work of Ruelle and was an inspiration for the general definition of topological pressure of Walters, see  \cite{ABL},  \cite[Theorem 13.1]{bar-kwa}.
Thus
\begin{equation*}
r(\pi(a)T)=r(L|a|^p)^{1/p}=\max_{\mu\in\Erg(X,\mu)}\exp\left( \int_X \ln (|a|\varrho^{\frac{1}{p}})d\mu + \frac{h_\varphi(\mu)}{p} \right).
\end{equation*}
This finishes the proof.

\bibliographystyle{spmpsci}

\begin{thebibliography}{10}
\providecommand{\url}[1]{{#1}}
\providecommand{\urlprefix}{URL }
\expandafter\ifx\csname urlstyle\endcsname\relax
  \providecommand{\doi}[1]{DOI~\discretionary{}{}{}#1}\else
  \providecommand{\doi}{DOI~\discretionary{}{}{}\begingroup
  \urlstyle{rm}\Url}\fi

\bibitem{Ando}
Ando T.: Contractive projections in {$L_p$} spaces.
\newblock Pac. J. Math. \textbf{17}(3), 390--405 (1966)

\bibitem{ABL}
Antonevich A., Bakhtin V., Lebedev A.: Ergodic theory and dynamical systems.
\newblock Ergod. Th. Dynam. Sys. \textbf{31}, 995--1042 (2011)

\bibitem{bar-kwa}
Bardadyn K., Kwa\'{s}niewski B.K.: Spectrum of weighted isometries:
  {$C^*$}-algebras, transfer operators and topological pressure.
\newblock Isr. J. Math. \textbf{246}(1), 149--210 (2021)

\bibitem{BKL}
Bardadyn K., Kwa\'sniewski B.K., Lebedev A.V.: {$C^*$}-algebras associated
  to transfer operators for countable-to-one maps (2023).
\newblock Preprint. arXiv:2202.03802

\bibitem{BKM}
Bardadyn K., Kwa\'sniewski B.K., McKee A.: Banach algebras associated to
  twisted \'etale groupoids: inverse semigroup disintegration and
  representations on ${L}^p$-spaces (2023).
\newblock Preprint. arXiv:2303.09997

\bibitem{BKM2}
Bardadyn K., Kwa\'sniewski B.K., McKee A.: Banach algebras associated to
  twisted \'etale groupoids {II}: simplicity and pure infiniteness (2024).
\newblock In preparation.

\bibitem{Bernau_Lacey_Elton}
Bernau S.~J., Lacey H.E.: The range of a contractive projection on an
  {$L_{p}$}-space.
\newblock Pacific J. Math. \textbf{53}, 21–41 (1974)

\bibitem{BJ}
Bratteli O., Jorgensen P.: Wavelets Through a Looking Glass: The World of the
  Spectrum, Appl. Numer. Harmon. Anal., vol.~2.
\newblock Birkhäuser Boston Inc., Boston, MA (2002)

\bibitem{cgt}
Choi Y., Gardella E., Thiel H.: Rigidity results for {$L^p$}-operator
  algebras and applications.
\newblock Preprint. arXiv: 1909.03612

\bibitem{Cuntz}
Cuntz J.: Simple {$C^*$}-algebras generated by isometries.
\newblock Comm. Math. Phys \textbf{57}, 173--185 (1977)

\bibitem{Exel_crossed}
Exel R.: A new look at the crossed-product of a {$C^\ast$}-algebra by an
  endomorphism.
\newblock Ergodic Theory Dynam. Systems \textbf{23}(6), 1733--1750 (2003)

\bibitem{Fremlin}
Fremlin D.H.: Measure theory. {V}ol. 2.
\newblock Torres Fremlin, Colchester (2003).
\newblock Broad foundations, Corrected second printing of the 2001 original

\bibitem{kwa1}
Kwa\'{s}niewski B.K.: Exel's crossed product and crossed products by
  completely positive maps

\bibitem{Lamperti}
Lamperti J.: On the isometries of certain function-spaces.
\newblock Pac. J. Math. \textbf{8}(3), 459--466 (1958)

\bibitem{Mbekhta}
Mbekhta M.: Partial isometries and generalized inverses.
\newblock Acta Sci. Math. (Szeged) \textbf{70}(3-4), 767--781 (2004)

\bibitem{PhLp1}
Phillips N.C.: Analogs of {C}untz algebras on {$L^p$}-spaces.
\newblock Preprint. arXiv: 1201.4196

\bibitem{PhLp2a}
Phillips N.C.: Simplicity of {UHF} and {C}untz algebras on {$L^p$}-spaces.
\newblock Preprint. arXiv: 1309.0115

\end{thebibliography}

\end{document}